\documentclass[12pt]{amsart}
\usepackage{amsmath,amscd,amssymb,amssymb}

\usepackage[mathcal]{eucal}

\usepackage{amsmath, amssymb, xy}
\input xy
\xyoption{all}
\usepackage{pdflscape}
\usepackage{hyperref}
\usepackage[vcentermath]{youngtab}
\usepackage{graphics}

\numberwithin{equation}{section}

\newtheorem{thm}{Theorem}[section]
\newtheorem{prop}[thm]{Proposition}

\newtheorem{cor}[thm]{Corollary}
\newtheorem{lem}[thm]{Lemma}

\theoremstyle{definition}
\newtheorem{defn}[thm]{Definition}
\newtheorem*{remark*}{Remark}

\newcommand{\cW}{{\mathcal W}}

\newcommand{\cP}{{\mathcal P}}

\newcommand{\bm}{{m\rho}}

\newcommand{\C}{\mathbb{C}}  
\newcommand{\Z}{\mathbb{Z}}  
\newcommand{\N}{\mathbb{N}}  

\newcommand{\lri}{\cO}  

\newcommand{\mi}{\cP}   

\newcommand{\ii}{i}
\newcommand{\jj}{j}

\newcommand{\GL}{\text{GL}}

\newcommand{\aGL}{\textup{\textsf{GL}}}
\newcommand{\aB}{\textup{\textsf{B}}}
\newcommand{\aT}{\textup{\textsf{T}}}

\newcommand{\Stab}{\text{Stab}}

\newcommand{\Hom}{\text{Hom}}

\newcommand{\cO}{\mathcal{O}}
\newcommand{\ind}{\mathrm{Ind}}

\newcommand{\val}{\mathrm{val}}

\newcommand{\am}{A}

\newcommand{\torusm}{\Theta}
\newcommand{\torusmij}{\torusm_{\ii\jj}}

\newcommand{\gammam}{\Gamma}

\newcommand{\level}{\lambda}

\newcommand{\minv}{\mu}
\newcommand{\kinv}{\kappa}

\newcommand{\htx}{\xi}
\newcommand{\htz}{\zeta}
\newcommand{\htxz}{{\htx,\htz}}

\newcommand{\com}{\cO_m}

\newcommand{\bmk}{B_m^{\delta}}
\newcommand{\emk}{E^\delta}
\newcommand{\dmk}{\Delta^\delta_m}

\newcommand{\ut}{\texttt{u}}
\newcommand{\st}{\texttt{s}}
\newcommand{\et}{\texttt{e}}
\newcommand{\Id}{\texttt{I}}
\newcommand{\M}{N^+}
\newcommand{\triv}{\mathbf{1}}

\newcommand{\cone}{{\textmd{C}}}
\newcommand{\cint}{{\cone^\circ}}
\newcommand{\cbd}{{\partial\cone}}

\title[Unramified principal series for $\GL(3)$]{On the unramified principal series of $\GL(3)$ over non-archimedean local fields}

\author{Uri Onn}\thanks{Onn was supported by the Israel Science Foundation (ISF grant 382/11)}\address{Department of Mathematics, Ben Gurion
  University of the Negev, Beer-Sheva 84105, Israel}
  \email{urionn@math.bgu.ac.il}

\author{Pooja Singla}\thanks{Singla was supported by the Center for Advanced Studies in Mathematics at Ben Gurion University}
\address{Department of Mathematics,
Indian Institute of Science,
Bangalore 560012, India}
  \email{pooja@math.iisc.ernet.in}



\begin{document}
\maketitle

\begin{abstract}
Let $F$ be a non-archimedean local field and let $\lri$ be its ring of integers.
We give a complete description of the irreducible constituents of the restriction of the unramified
principal series representations of $\aGL_3(F)$ to $\aGL_3(\lri)$.
\end{abstract}

\section{Introduction}\label{sec:intro}

\subsection{Overview}
Let $F$ be a non-archimedean local field with ring of integers $\lri$, maximal ideal $\mi$ and let $\pi$ be a fixed uniformizer. Our interest in this paper is in the restriction of unramified principal series representations of $\aGL_n(F)$ to its maximal compact subgroup $\aGL_n(\lri)$. More precisely, let $\aB$ be a Borel subgroup of $\aGL_n$ and let $\aT$ be a maximal torus contained in $\aB$. Concretely, we may take the group of upper triangular matrices and the group of diagonal matrices, respectively. Any linear character $\chi: \aT(F) \to \C^\times$ can be inflated to a linear character of $\aB(F)$, still denoted $\chi$. The principal series representation corresponding to $\chi$ is the induced representation $\ind_{\aB(F)}^{\aGL_n(F)}(\chi)$ of $\aGL_n(F)$ on the space of continuous functions
\[
V_\chi=\{ f \in C(\aGL_n(F)) \mid f(gb)=\chi(b)\| b \|^{1/2} f(g), \text{for all $g \in \aGL_n(F), b \in \aB(F)$}   \},
\]
where $b \mapsto \|b\|$ is the modular character of $\aB(F)$. We assume that the representation is unramified, namely, that the restriction of $\chi$ to $\aT(\lri)$ is trivial, and focus on the decomposition to irreducible constituents of its restriction to $\aGL_n(\lri)$. By Frobenius reciprocity this restriction is isomorphic to $\ind_{\aB(\lri)}^{\aGL_n(\lri)}(\triv)$. The case $n=2$ was fully treated in \cite{MR0338274}. Partial results for the case $n=3$ were obtained in \cite{MR2504482} and it is the goal of this paper to give a complete description in this case. Further results on the restriction of principal series representations of $\GL_n$ to the maximal compact subgroup can be found in \cite{0257.22018}.



\subsection{The unramified principal series of $\GL(3)$}\label{subsec:CN} We first describe and reformulate several results concerning $\GL_3$ which were obtained by Campbell and Nevins in~\cite{MR2504482}. Let $G = \aGL_3(\cO)$, let $B$ be its subgroup of upper triangular matrices and let $V = \ind_{B}^G(\triv)$. For $\ell \in \N$, let $K^{\ell}$ denote the $\ell^{\text{th}}$ principal congruence subgroup of $G$, namely, the kernel of the canonical map from $G$ to $\aGL_3(\cO/\mi^{\ell})$. Being normal subgroups of $G$, the groups $K^\ell$ give rise to a coarse decomposition of $V$ into $G$-invariant spaces
$ V \cong \oplus_{\ell=1}^{\infty} V^{K^\ell}/V^{K^{\ell-1}}$. A finer decomposition of these spaces, yet not to irreducible ones, is obtained by considering certain \lq parabolic\rq ~ subgroups of $G$ which are defined as follows. Let $\N_0$ stand for $\N \cup \{0\}$ endowed with the natural ordering and consider $\N_0^3$ with the product ordering, i.e. $c \preccurlyeq d$ if and only if $c_i \le d_i$ for $i=1,2,3$. Let $\cone$ be the cone
\[
\cone = \{(c_1, c_2, c_3) \in \N_0^3 \mid  c_1, c_2 \leq c_3 \leq c_1 + c_2 \}.
\]
To an element $c = (c_1, c_2, c_3)\in \cone$ we associate the compact open subgroup of $G$
\[
P_c = \left[
\begin{matrix}
\cO & \cO & \cO \\
\mi^{c_1} & \cO & \cO \\
\mi^{c_3} & \mi^{c_2} & \cO
\end{matrix}
\right] \cap G,
\]
which contains $B$. The defining inequalities of $\cone$ ensure that $P_c$ is indeed a group. Let $U_c = \ind_{P_c}^G(\triv)$ be the permutation representation of $G$ arising from its action on the coset space $G/P_c$. It follows that $P_c \subseteq P_d $ if and only if $c \succcurlyeq d$, hence $U_d$ is a subrepresentation of $U_c$ if and only if $d \preccurlyeq c$.
In the center of our discussion lie the representations
\[
V_c = U_c/\sum_{d \prec c} U_d, \qquad (c \in \cone).
\]
We have
\begin{eqnarray}
\label{eq:Vc-decomposition}
V =  \bigoplus_{c \in \cone} V_c.
\end{eqnarray}
The mutual relations between the representations $V_c$ are completely determined in \cite{MR2504482}. To describe them we introduce the auxiliary functions $\level,\kinv,\minv :\cone \to \N_0$. For $c \in \cone$ let
\begin{align*}
\level(c)&=c_3, \\
\kinv(c)&=c_1+c_2-c_3,\\
\minv(c)&=\min \{c_1+c_2-c_3, c_3-c_1,c_3-c_2\}.
\end{align*}
With these we have
\begin{align*}\label{cn1}\tag{CN1}
&\text{Let $c,d \in \cone$. Then} \\ &\text{$V_c \simeq V_{d}$ if and only if $\level(c)=\level(d)$, $\kinv(c)=\kinv(d)$, $\minv(c)=\minv(d)$ and $c=d$ if $\kinv(c)>\minv(c)$}. \qquad
\end{align*}
Statement~\eqref{cn1} is a reformulation of several results from~\cite{MR2504482} which requires a proof (see Proposition~\ref{prop:reformulation} below). Let $\cint$ and $\cbd$ denote the interior and boundary of $\cone$, respectively. That is
\[
\begin{split}
\cint &=  \{ c = (c_1, c_2, c_3) \mid     c_1,c_2 <   c_3 < c_1 + c_2 \}, \\
 \cbd &= \cone \smallsetminus \cint.
\end{split}
\]
Concerning the irreducibility of the representations $V_c$ one has
\begin{align}\label{cn2}   \tag{{CN2}}
&\text{The representation $V_c$ is irreducible if and only if $c \in \cbd$.} \qquad \qquad & \qquad & \qquad \qquad
\end{align}


Statement~\eqref{cn2} is proved in \cite[Theorems 6.1, 7.1 and 8.1]{MR2504482}. The complete description of the double coset spaces $P_c \backslash G  / P_d$ ($c,d \in \cone$), obtained in~\cite{MR2504482}, does not lend itself to decompose the representations $V_c$ ($c \in \cint$), which comprise most of the constituents of $V$. However, crucial for us is the following observation made in ~\cite[\S8]{MR2504482}. Let $\rho=(1,1,1)$. Let $m \in \N$ and $c \in \cint \cap \cone+\bm$. For all $d \in [c-\bm,c]$ let
\begin{equation}\label{eq:def.of.Udm.Vdm}
\begin{split}
U_d^{m} &= \mathrm{Ind}_{P_d}^{P_{c-\bm}}(\triv), \\
V_c^{m} &= U_c^{m}/\sum_{c-\bm \preccurlyeq d \prec c} U_d^{m}.
\end{split}
\end{equation}
It follows that $V_c = \mathrm{Ind}_{P_{c-\bm}}^{G} \left(V_c^{m}\right)$. For the special case $m=\mu(c)$ the relations between $V_c$ and $V_c^m$ become very tight.
\begin{align*}\label{cn3}\tag{CN3}
&\text{Let $c \in \cint$. Then $\dim \mathrm{End}_G\left( V_c \right) = \dim \mathrm{End} _{P_{c-\minv(c)\rho}}\left(V_c^{\minv(c)}\right)$. In particular, the irreducible}\\ & \text{constituents of $V_c$ are induced from the irreducible constituents of $V_c^{\mu(c)}$. }
\end{align*}

We remark that any linear combination of the invariants $\level$, $\kinv$ and $\minv$ above could be used throughout. We made this particular choices as they have natural interpretations: $\level(c)$ is the level of the representation $V_c$, and both $\kinv(c)$, $\minv(c)$ turn out to be the parameters which control the decomposition of $V_c$ ($c \in \cint$).

\subsection{Description of results and organization of the paper} The main result of this paper is a complete description of the irreducible constituents of the representations $V_c$ ($c \in \cint$). These are described in terms of induced linear characters of certain subquotients of the groups $P_c$.
Consequently we obtain a complete description of the irreducibles in $V$. In Section~\ref{sec:bmk} we define twisted Heisenberg groups and their toral extensions. As it turns out, these are quotients of $P_c$ ($c \in \cint$) which carry the complete information on the representations $V_c$ but are much more accessible. In Section~\ref{sec:bmk-mod-torus} we define and analyze in detail specific multiplicity free representations of toral extensions of twisted Heisenberg groups and in Section~\ref{sec:embedding} we show that the $V_c^m$'s for $m\in \mathbb N$ and $\mu(c) \geq m$ are subrepresentations of these multiplicity free representations. It remains to identify what are the components of $V_c^m$ and this is achieved in Section \ref{sec:decomposition}. We end the paper by computing the multiplicities and dimensions of the irreducible constituents in $V$. Notably, these depend only on the residue field $\lri/\mi$, and the dependence is through substitution in universal polynomials defined over $\Z$.

\subsection{Notation, conventions and tools}
For $m \in \N$ we write $\lri_m$ for the finite quotient $\lri/\mi^m$. We use $\val(\cdot)$ to denote the valuation on $\lri$. It is convenient to use the same symbol for the finite quotients with the convention that $\val(0)=m$ for $0 \in \lri_m$. When chances for confusion are slim we use the same notation for elements or subsets of $\lri$ and their respective images in $\com$. For a group $G$ and elements $g,x \in G$ we write ${}^gx=gxg^{-1}$ for the left conjugation action and $x^g=g^{-1}xg$ for right conjugation. If $H < G$ we write ${}^gH=gHg^{-1}$. If $\chi$ is a character of $H$ we write ${}^g\chi$ for the character of ${}^gH$ defined by ${}^g\chi(x)=\chi(g^{-1}xg)$ for all $x \in {}^gH$. Throughout we use fairly standard  tools from representation theory such as Mackey and Clifford theories and the following criterion for the existence of non-trivial intertwining operators. If $\chi_i$ are linear characters of subgroups $H_i$ of $G$ we have the following realization of the intertwining operators
\[
\Hom_G\left(\ind_{H_1}^{G}(\chi_1), \ind_{H_2}^{G}(\chi_2)\right)=\left\{f:G \to \C \mid f(h_2gh_1)=\chi_2(h_2)f(g)\chi_1(h_1), \forall h_i \in H_i, \forall g \in G\right\},
\]
and an element $g \in G$ supports a non-zero intertwining function if and only if $\chi_1={}^g\chi_2$ on $H_1 \cap {}^gH_2$.

\section{Twisted Heisenberg groups and their toral extensions}\label{sec:bmk}

To study the representation $V_c^m$ we first need to analyze $U_c^m$, which is the permutation representation of $P_{c-\bm}$ arising from its action on the right cosets $P_{c-\bm}/P_c$ for $m \in \N$, $c \in \cint \cap \cone+\bm$.

\begin{defn} Let $R$ be a commutative ring with identity and $\delta \in R$. The {\em $\delta$-twisted Heisenberg groups over $R$}, denoted $H^\delta_R$, is the set of triples in $R^3$ endowed with the multiplication
\[
(x,y,z)\cdot(x',y',z') = (x+x', y+y', z+z'+ \delta y x').
\]

A {\em toral extension of a $\delta$-twisted Heisenberg group} is the semidirect product $H^\delta_R \rtimes  (R^\times)^3$ of $(R^\times)^3$ and $H^\delta_R$, with the action
\begin{equation}\label{eq:action.of.torus}
{}^t(x,y,z)=(t_1^{-1}t_2x,t_2^{-1}t_3y,t_1^{-1}t_3z),
\end{equation}
for $t=(t_1,t_2,t_3) \in (R^\times)^3$ and $(x,y,z) \in H^\delta_R$. We denote this semidirect product by $B^\delta_R$.
\end{defn}

The family $\{ H^\delta_R \mid \delta \in R\}$ interpolate between the standard Heisenberg group ($\delta=1$) and the abelian group $R^3$ with addition ($\delta = 0$). The group $B^{\delta=1}_R$ is isomorphic to the group of invertible triangular matrices over $R$. It is convenient to keep the analogy with this special case and use a matrix presentation for $B^\delta_R$:
\begin{equation}\label{eq:matrix.presentation.bmk}
[(x,y,z),(t_1,t_2,t_3)] \mapsto \left[\begin{matrix} 1 &  &  \\ x & 1 &  \\ z & y & 1 \end{matrix} \right] \circ_\delta\left[\begin{matrix} t_1 &  &  \\  & t_2 &  \\  &  & t_3 \end{matrix} \right],
\end{equation}
with multiplication defined by
\begin{equation}\label{eq:multiplication.bkm}
\left[\begin{matrix} t_1 &  &  \\ x & t_2 &  \\ z & y & t_3 \end{matrix} \right] \circ_\delta \left[\begin{matrix} t'_1 &  &  \\ x' & t'_2 &  \\ z' & y' & t'_3 \end{matrix} \right]=\left[\begin{matrix} t_1t'_1 &  &  \\ x t'_1+t_2x' & t_2t'_2 &  \\ zt'_1+\delta yx'+t_3z' & yt'_2+t_3y' & t_3t'_3 \end{matrix} \right],
\end{equation}
which is almost the usual matrix multiplication except for the twist by $\delta$ at the $(3,1)$-entry. It will be useful in the sequel to have the following description of $B_R^\delta$ modulo its center.

\begin{prop}\label{prop:stucture.of.bmk.mod.center} Let $R$ be a commutative ring with identity and $\delta \in R$. Let
\[
 E^\delta_R = R^2  \rtimes_\delta \left[ \begin{matrix} R^\times & \\ R & R^\times \end{matrix}\right],
\]
with the action
\begin{equation}\label{eq:action.delta}
{\left[ \begin{matrix} \alpha & \\ \beta & \gamma \end{matrix}\right]} \circ_\delta \left[\begin{matrix} x \\ z \end{matrix}\right]=\left[ \begin{matrix}
\alpha x \\ \beta x\delta+\gamma z \end{matrix}\right].
\end{equation}
Then $B^\delta_R/Z_{B^\delta_R} \simeq E^\delta_R$.

\end{prop}
\begin{proof} The explicit isomorphism is
\begin{equation}\label{eq:bmktoemk}
\phi:\left[\begin{matrix} t_1 &  &  \\ x & t_2 &  \\ z & y & t_3 \end{matrix} \right] \mapsto t_1^{-1}\left[\begin{matrix} x \\ z \end{matrix}\right] \rtimes t_1^{-1}\left[ \begin{matrix} t_2  & \\ y & t_3 \end{matrix}\right]
\end{equation}
\end{proof}


Let $R=\lri/\mi^m=\com$ and in this case denote the groups above $H^\delta_m$, $\bmk$ and $\emk$. Let $T_m$ stand for the subgroup of diagonal matrices in $\bmk$. The relevance of the groups $\bmk$ and $\emk$ to the permutation representations $U_c^m$ comes from Theorem~\ref{thm:identification-of-actions}.

For $c \in \cone$ let
\[
N_c = \left[ \begin{matrix} 1 & & \\ \mi^{c_1} & 1 & \\ \mi^{c_3} & \mi^{c_2} & 1 \end{matrix}\right],
\M = \left[ \begin{matrix}
1 & \cO & \cO \\ & 1 & \cO \\ & & 1
\end{matrix} \right],
T^m= T \cap K^{m}.
\]
\begin{lem} Let $c=(c_1,c_2,c_3)  \in \cone$ and let $m \in \mathbb N$ such that $\mu(c) \geq m$. Let $\delta = \pi ^{\kinv(c-m\rho)}=\pi^{c_1+c_2-c_3-m}$. Then the map $\eta : P_{c-m\rho} \to (\bmk,\circ_\delta)$ defined by
\begin{equation}\label{eq:Pcmktobmk}
\eta:\left[\begin{matrix} t_1 & \star & \star \\ \pi^{c_1-m}x & t_2 & \star \\ \pi^{c_3-m}z & \pi^{c_2-m}y & t_3 \end{matrix} \right]
\mapsto \left[\begin{matrix} t_1 &  &  \\ x & t_2 &  \\ z & y & t_3 \end{matrix} \right] \mod \mi^m,
\end{equation}
is a surjective homomorphism with kernel $N_{c} T^m N^{+}$.
\end{lem}
\begin{proof} First, note that the condition $\mu(c)\ge m$ implies that $c-m\rho \in \cone$, hence $P_{c-m\rho}$ is indeed a group. Clearly, $\eta$ is surjective and a straightforward matrix multiplication shows that $\eta$ is a homomorphism if and only if 
\begin{itemize}
\item [(a)] $c_1-m,c_2-m,c_3-m \ge m$, and
\item [(b)] $c_3-c_1,c_3-c_2,c_1+c_2-c_3 \ge m$.
\end{itemize} 
The condition (b) is equivalent to $\mu(c) \ge m$ and (a) follows from (b). Finally, it is clear that $N_{c} T^m N^{+}$ is contained in the kernel of $\eta$ and a short matrix computation shows equality. 
\end{proof}

\begin{thm}
\label{thm:identification-of-actions} Let $c=(c_1, c_2, c_3)  \in \cone $ and let $m \in \mathbb N$ such that $\mu(c) \geq m$. Let $\delta = \pi ^{\kinv(c-\bm)} = \pi^{c_1 + c_2 -c_3 -m }$. Then
\begin{itemize}
\item[(a)] The action of $P_{c-\bm}$ on $P_{c-\bm}/P_{c}$ factors through $\bmk$ via $\eta$.
\item[(b)] As $\bmk$-spaces $P_{c-\bm}/P_{c} \simeq \bmk/T_m$. In particular,  we have
\[
U_{c}^m \simeq \ind_{T_m}^{\bmk}(\triv),
\]
as $\bmk$-representations.
\end{itemize}
\end{thm}

Before proving Theorem~\ref{thm:identification-of-actions} we setup notation and make some preparation.
Let $\et_{ij}$ denote the matrix whose $(i,j)^{\text{th}}$ entry equal one and the rest of entries are zero. For $i \ne j$ and $x\in R$ (a commutative ring) let
\begin{equation}\label{elementary.elements.def}
\begin{split}
\ut_{ij}(x) &= \Id + x\et_{ij}, \qquad \text{(elementary unipotent)}, \\
\st_{i}(x)& =\Id+(x-1)\et_{ii}, \qquad \text{(elementary semisimple)}.
\end{split}
\end{equation}
The following are well known (and easily verified).
\begin{equation}\label{conjugation.unipotent.formula}
\ut_{ij}(x) \ut_{kl}(y)\ut_{ij}(x)^{-1} =\Id + y\et_{kl} +\left\{
                                     \begin{array}{ll}
                                       0, & \hbox{if $i\ne l,j \ne k$;} \\
                                       -xy \et_{kj}, & \hbox{if $i = l,j \ne k$;} \\
                                       xy \et_{il}, & \hbox{if $i\ne l,j = k$;} \\
                                       xy(\et_{ll}-\et_{kk})-x^2y \et_{ij}, & \hbox{if $i = l,j = k$.}
                                     \end{array}
                                   \right.
\end{equation}

for all $x,y \in R$, $i\ne j$ and $k \ne l$.

\begin{equation}\label{commutator.semisimple.unipotent}
[\ut_{ij}(x),\st_k(y)]=\left\{
                     \begin{array}{ll}
                       \Id, & \hbox{if $i\ne k, j\ne k$;} \\
                       \ut_{ij}(x(1-y)), & \hbox{if $i=k$;} \\
                       \ut_{ij}(x(1-y^{-1})), & \hbox{if $j=k$,}
                     \end{array}
                   \right.
\end{equation}
for all $x \in R$, $y \in R^\times$.
\begin{lem}
\label{lem: decomposition of parabolic} For all $c \in \cone$ and $m \in \mathbb N$ with $\mu(c) \geq m$ the following hold.
\begin{enumerate}
\item $P_{c-\bm} = N_{c-\bm} T \M$.
\item $N_{c} \lhd N_{c-\bm}$.
\item $[T^m,N_{c-\bm}] \subset N_{c}$.
\end{enumerate}
\end{lem}
\begin{proof} \qquad
\begin{enumerate}
\item The map $N_{c-\bm}\times  T \times  \M \to P_{c-\bm}$ given by $(n,t,n^+) \mapsto ntn^+$ is a bijection.

\item Since both groups are generated by elementary unipotents \eqref{elementary.elements.def} it is enough to check that $g n g^{-1} \in N_{c}$ for elementary unipotents $g \in N_{c-\bm}$, $n \in N_{c}$. There are only two pairs of such elements which do not commute and for them we verify using \eqref{conjugation.unipotent.formula} that
\[
\begin{split}
\ut_{21}(\pi^{c_1-m}x) \ut_{32}(\pi^{c_2}y)\ut_{21}(\pi^{c_1-m}x)^{-1}& = \Id + \pi^{c_2}y\et_{32} - \pi^{c_1+c_2-m}xy \et_{31} \in N_{c}, \\
\ut_{32}(\pi^{c_2-m}x) \ut_{21}(\pi^{c_1}y)\ut_{32}(\pi^{c_2-m}x)^{-1}&=\Id + \pi^{c_1}y\et_{21} + \pi^{c_1+c_2-m}xy \et_{31} \in N_{c}.
\end{split}
\]

\item Follows immediately from \eqref{commutator.semisimple.unipotent}.

\end{enumerate}

\end{proof}

\begin{prop}\label{prop:Ncm-action}
 For all $c \in \cone $ and $m \in \mathbb N$ with $\mu(c) \geq m$ the following hold.
\begin{enumerate}
\item As $N_{c-\bm}$-spaces $P_{c-\bm}/P_{c} \simeq N_{c-\bm}T/N_{c} T\simeq  N_{c-\bm}/N_{c}$.
\item The group $N_{c} T^m \M$ is contained in the kernel of the $P_{c-\bm}$-action on $P_{c-\bm}/P_{c}$.
\end{enumerate}
\end{prop}
\begin{proof}
(1) By Lemma~\ref{lem: decomposition of parabolic}, the group $N_{c-\bm}$ acts transitively on the left cosets of $P_{c}=N_{c} T \M$ in $P_{c-\bm}=N_{c-\bm}T \M$. The stabilizer of $P_{c}$ under this action is $N_{c}$. It follows that, as $N_{c-\bm}$-spaces, $P_{c-\bm}/P_{c} \simeq N_{c-\bm}/N_{c}$. This in particular implies that every element in $P_{c-\bm}/P_{c}$ can be represented as $n N_{c} T \M$ for some $n \in N_{c-\bm}/N_{c}$.

(2) It is clear from Lemma~\ref{lem: decomposition of parabolic}, that both $T^m$ and $N_{c}$ act trivially on $ P_{c-\bm}/P_{c}$. We show that $\M$ acts trivially as well. We need to show that $h(n N_{c}T \M)=n N_{c} T \M$ for $h \in \M$ and $n \in N_{c}$, or equivalently that $nhn^{-1} \in N_{c} T \M$. Since both groups $N_{c-\bm}$ and $\M$ are generated by unipotents it is enough to check that
\[
\ut_{ij}(x) \ut_{kl}(y)\ut_{ij}(x)^{-1} \in N_{c}T \M,
\]
for all $\ut_{ij}(x) \in N_{c-\bm}$ and $\ut_{kl}(y) \in \M$. The possible indices are $(i,j) \in \{(2,1),(3,1),(3,2)\}$ and $(k,l) \in \{(1,2),(1,3),(2,3)\}$. This is straightforward using \eqref{conjugation.unipotent.formula}. Thus $N_{c} T^m \M$ is contained in the  kernel.
\end{proof}

\begin{proof}[Proof of Theorem~\ref{thm:identification-of-actions}]
Using Lemma~\ref{lem: decomposition of parabolic} and Proposition~\ref{prop:Ncm-action} we have a commutative diagram
\[
\begin{matrix}
  N_{c}  & \times  & T^m  & \times & \M & \longrightarrow &  N_{c} T^m \M  \\
  \downarrow &    &  \downarrow & & \downarrow & & \downarrow \\
  N_{c-\bm}  & \times  & T  & \times & \M & \longrightarrow &  P_{c-\bm} \\
  \downarrow &    &  \downarrow & & \downarrow & & \downarrow \\
  N_{c-\bm}/ N_{c}  & \times  & T_m  & \times & \{\Id\} & \longrightarrow &  P_{c-\bm}/N_{c} T^m \M \\
  \wr| &    &  || & & || & & \wr| \\
  H^\delta_m  & \times  & T_m  & \times & \{\Id\} & \longrightarrow &  \bmk, \\
\end{matrix}
\]
with vertical maps being embeddings or canonical epimorphisms and horizontal bijective maps given by multiplication. The identification of the last two rows is induced by the map $\eta : P_{c-\bm} \to (\bmk,\circ_\delta)$. Assertions (a) and (b) follow.

\end{proof}

For $r=(r_1,r_2,r_3) \in \N_0^3$ with $r_3 \le r_1+r_2$ let
\begin{equation}\label{eq:QrNr}
N^\delta_{r}=\left[\begin{matrix} 1 &  &  \\ \mi^{r_1} & 1  &  \\ \mi^{r_3} &\mi^{r_2} & 1 \end{matrix} \right] \!\!\!\!\mod \mi^m, \qquad Q^\delta_{r}=\left[\begin{matrix} \lri^\times &  &  \\ \mi^{r_1} & \lri^\times  &  \\ \mi^{r_3} &\mi^{r_2} &\lri^\times  \end{matrix} \right] \!\!\!\!\mod \mi^m
\end{equation}
considered as a subgroup of $\bmk$. As $N^\delta_r$ does not intersect the center of $\bmk$ we also let $N^\delta_r$ denote its image in $\emk$. Let $\Theta$ denote the image of $T_m$ in $\emk$.

\begin{cor}\label{cor:Bd-c}  For all $c \in \cone $, $m \in \mathbb N$ with $\mu(c) \geq m$ and $d$ such that $c - m\rho \preccurlyeq d \preccurlyeq c $ we have
\[
U_d^m \simeq \ind_{Q^\delta_{d-c + \bm}}^{\bmk} (\triv),
\]
as $\bmk$-representations.
\[
U_d^m \simeq \ind_{N^\delta_{d-c + \bm}\Theta}^{\emk} (\triv),
\]
as $\emk$-representations. In particular, $U_{c}^m \simeq \ind_{\Theta}^{\emk} (\triv)$.

\end{cor}

\begin{proof} The first part follows from the proof of Theorem~\ref{thm:identification-of-actions} observing that $\eta(P_d)=Q^\delta_{d-c + \bm}$. The second part follows by observing that the center of $\bmk$ acts trivially.

\end{proof}


\section{On the permutation representation $\C[\emk/\torusm]$}\label{sec:bmk-mod-torus}
Let $E^\delta=\am \rtimes_\delta \gammam$, with
\[
\begin{split}
\am &= \left[ \begin{matrix} \com \\ \com \end{matrix}\right], \\
\gammam &=\left[ \begin{matrix} \com^\times & \\ \com & \com^\times \end{matrix}\right],
\end{split}
\]
and the action given by \eqref{eq:action.delta}. Let $\torusm$ stand for diagonal matrices in $\gammam$
\[
\torusm=\left[ \begin{matrix} \com^\times & \\ & \com^\times \end{matrix}\right].
\]

The aim of this section is to construct and analyze certain subrepresentations of the permutation representation  $\ind_{\torusm}^{\emk}(\triv)=\C[\emk/\torusm]$. To simplify notation we identify $\am$ and $\gammam$ with their images in $\emk$.
We start with the permutation representation of $\am\torusm$ on $\C[\am\torusm/\torusm]$. The group $\am\torusm$ is isomorphic to $a \in \am, \theta \in \torusm$. Let $\psi: \com \to \C^\times$ be a character which establishes the self duality $\com \simeq \widehat{\com}$, that is, every character of $\com$ is of the form $\psi_{\htx}:x \mapsto \psi(\htx x)$ for some $\htx \in \com$. The characters of $A$ are parameterized by pairs $a=(\htx,\htz) \in \com^2$ which correspond to the character $\varphi_{\htx,\htz}:(x,z) \mapsto \psi(\htx x+\htz z)$. Using the induced action of $\torusm$ on $\widehat{\am}$ we see that the orbits of $\torusm$ on $\widehat{A}$ are
\[
\Omega_{\ii\jj}=\{ \varphi_{\htx,\htz} \mid \val(\htx)=\ii, \val(\htz)=\jj\}, \qquad 0 \le \ii,\jj \le m.
\]
Fix $0 \le i,j \le m$ and $\varphi \in \Omega_{\ii\jj}$. The stabilizer of $\varphi$ in $\torusm$ is
\[
\begin{split}
\torusm_{\ii\jj}=\Stab_{\torusm}(\varphi)&=\{\mathrm{diag}(\theta_1,\theta_2) \mid  \theta_1 \equiv_{\pi^{m-\ii}} 1 , \,\,\theta_2\equiv_{\pi^{m-\jj}} 1 \} \\& \simeq 1+\pi^{m-\ii}\com \times 1+\pi^{m-\jj}\com,
\end{split}
\]
and we can extend $\varphi$ to a character of $A\torusm_{\ii\jj}$ using the trivial representation of $\torusm_{\ii\jj}$. We denote this extension by  $\varphi'$. Let
\begin{equation}\label{eq:def.of.Wij}
\cW_{\ii\jj}=\bigoplus_{\varphi \in \Omega_{\ii\jj}} \varphi'.
\end{equation}
It follows that $\cW_{\ii\jj}$ is an irreducible representation of $\am\torusm$ which is isomorphic to the induced representation $\ind_{\am\torusmij}^{\am\torusm}(\omega')$ for every $\omega \in \Omega_{\ii\jj}$.
It is the unique irreducible representation of $\am$ lying above the orbit $\Omega_{\ii\jj}$ which has a $\torusm$-fixed vector. Indeed, since $\torusm \backslash \am\torusm / \am \torusmij$ is a singelton, the space
\[
\Hom_{\am\torusm}\left(\cW_{ij},\ind_{\torusm}^{\am\torusm}(\triv)\right)\simeq\left\{f:\am\torusm\to \C \mid f(\theta g h)=\triv(\theta)f(g)\varphi'(h), \,\text{$\forall\theta \in \torusm$, $\forall g \in \am\torusm$, $\forall h \in \am\torusmij$}\right\}
\]
is at most one-dimensional. On the other hand a non-trivial intertwining function does exist since $\varphi=\triv$ on $\am\torusmij \cap \torusm = \torusmij$. We record the last discussion in

\begin{lem}\label{lem:indTA1} There exists decomposition to irreducible $\am\torusm$-representations
\[
\ind_{\torusm}^{\am\torusm}(\triv)=\bigoplus_{0 \le i,j \le m} \cW_{ij},
\]
where $\cW_{ij}$ is the unique irreducible representation of $\am\torusm$ which contains a $\torusm$-fixed vector and whose restriction to $\am$ is the sum of the characters in the orbit $\Omega_{\ii\jj}$. In particular, $\cW_{\ii\jj} \simeq \ind_{\am\torusmij}^{\am\torusm}(\varphi')$ for every $\varphi \in \Omega_{\ii\jj}$, and
\[
\dim  \cW_{\ii\jj}= |\Omega_{\ii\jj}|=|\pi^\ii \com^\times||\pi^\jj \com^\times|=q^{2m-2-\ii-\jj}(q-1)^2.
\]

\end{lem}

We now induce the representations $\cW_{ij}$ further to $\emk$, and define
\begin{equation} \label{eq:def.of.tildeWij}
\widetilde{\cW}_{ij}=\ind_{\am\torusm}^{\emk}(\cW_{ij}).
\end{equation}
Our next goal is to find the irreducible components of $\widetilde{\cW}_{00}$.

\begin{lem}\label{lem:stab.of.varphi.in.emk} Let $\varphi_{\htx,\htz} \in \Omega_{00}$ and denote $\epsilon=\delta\htx^{-1}\htz$. The stabilizer of $\varphi_{\htxz}$ in $\gammam$ is
equal to $\Delta_m^\epsilon$, where
\begin{enumerate}
\item [(a)] If $\delta \in \com^\times$
\[
\Delta^\epsilon_m=\left\{ \left[ \begin{smallmatrix} \alpha & \\\beta & 1 \end{smallmatrix}\right] \mid \alpha \in \com^\times, \beta = \epsilon^{-1}(1-\alpha) \right\} \simeq \com^\times.
\]
\item [(b)] If $\delta \in \pi\com$
\[
\Delta^\epsilon_m=\left\{ \left[ \begin{smallmatrix} \alpha &  \\ \beta & 1 \end{smallmatrix}\right] \mid \beta \in \com, \alpha=1-\epsilon \beta \right\} \simeq \com.
\]
\end{enumerate}
\end{lem}

\begin{proof} Using the action \eqref{eq:action.delta} we have
\[
{}^{\left[\begin{smallmatrix} \alpha & \\ \beta & \gamma \end{smallmatrix} \right]}\varphi_{\htx,\htz}=\varphi_{\alpha\htx+ \delta \beta\htz,\gamma\htz},
\]
hence the stabilizer consists of elements $\left[\begin{smallmatrix} \alpha & \\ \beta & \gamma \end{smallmatrix} \right]$ whose entries solve the equations
\[
\begin{split}
\htx&=\alpha\htx+ \delta \beta\htz \\
\htz&=\gamma\htz.
\end{split}
\]
The solution depends on $\delta$ being a unit or not and is given by cases (a) and (b).

\end{proof}

\begin{thm}\label{thm:structure.tilde.W00}  The representation $\widetilde{\cW}_{00}$ of $\emk$ is multiplicity free with equidimensional irreducible constituents. More precisely,  $\widetilde{\cW}_{00}$ has a decomposition
\[
\widetilde{\cW}_{00} \simeq \bigoplus_{\sigma \in \Sigma} L_\sigma,
\]
where
\[
\Sigma=\left\{
                 \begin{array}{ll}
                   \widehat{\com^\times}, & \hbox{{if $\delta \in \com^\times$;}} \\
                   \widehat{\com}, & \hbox{if $\delta \in \pi\com$.}
                 \end{array}
               \right.
\]
The representations $L_\sigma$ are irreducible, non-equivalent and are induced from the one-dimensional extensions of $\varphi=\varphi_{1,1}$ to    $\mathrm{Stab}_{\emk}(\varphi)=A\dmk$. In particular, their dimension is
\[
\dim L_\sigma=[\emk: \mathrm{Stab}_{\emk}(\varphi)] =\left\{
                 \begin{array}{ll}
                   q^{2m-1}(q-1), & \hbox{{if $\delta \in \com^\times$;}} \\
                   q^{2m-2}(q-1)^2, & \hbox{if $\delta \in \pi\com$.}
                 \end{array}
               \right.
\]

\end{thm}

\begin{proof}

Using Lemma~\ref{lem:indTA1} we have $\widetilde{\cW}_{00} \simeq \ind^{E^\delta}_{\am}(\varphi)$ for any $\varphi=\varphi_{\htx,\htz}$ with $\htx,\htz \in \com^\times$ which we may specify as $\htx=\htz=1$. By Lemma~\ref{lem:stab.of.varphi.in.emk} we have $\Stab_{\emk}(\varphi)=A\dmk$. Being a semidirect product, the characters of $A\dmk$ which extend $\varphi$ are of the form $\varphi\sigma$ where $\sigma \in \Sigma=\widehat{\dmk}$. It follows that
\begin{equation}\label{eq:def.of.Llmambda}
L_\sigma=\ind_{A\dmk}^{E^\delta}(\varphi \sigma), \quad (\sigma \in \Sigma),
\end{equation}
are irreducible and distinct. By Clifford's theorem, the representations $L_\sigma$ are precisely the irreducible constituents of $\ind^{\emk}_{\am}(\varphi)$, perhaps with multiplicities. However, their direct sum is of dimension $[\emk:A]$, therefore, they occur with multiplicity one.


\end{proof}


\section{An embedding}\label{sec:embedding}
The main result in this section is the following.
\begin{thm}\label{thm:embedding} For all $c \in \cone $ and $m \in \mathbb N$ with $\mu(c) \geq m$, the representation $V_c^m$ is a subrepresentation of $\widetilde{\cW}_{00}$. In particular, $V_c^m$ is multiplicity free.
\end{thm}

Later on we shall prove a finer result (Theorem~\ref{thm:LsigmaInVcm}) in which the precise decomposition of $V_c^m$ is given. Nevertheless, the arguments needed to prove Theorem~\ref{thm:embedding} are softer and seem to be in a better position towards generalization to $\GL_n$ with $n > 3$.

\begin{prop}\label{prop:Udm-subrep-TildeWij} Let $c \in \cone $ and let $ m \in \mathbb N$ such that $\mu(c) \geq m$. Then, for $0 \le i,j \le m$, we have  $\widetilde{\cW}_{ij}  \subset U^m_{d}$ if one (or both) of the following hold.
\begin{enumerate}
\item [(a)] $i>0$ and  $d=(c_1-1,c_2,c_3)$.

\item [(b)] $j>0$ and $d=(c_1,c_2,c_3-1)$.
\end{enumerate}
\end{prop}

\begin{proof} Let $i,j$ and $d$ satisfy (a) or (b). Note that the assumptions on $c$ and $m$ imply that
$d \in \cone$ and that $c-\bm \prec d \prec c$. Let $r=d-(c-\bm)$. By Lemma~\ref{lem:indTA1} and Corollary \ref{cor:Bd-c}
we have
\[
\begin{split}
\widetilde{\cW}_{ij} &\simeq \ind^{\emk}_{\am\torusm}(\cW_{ij}) \simeq \ind^{\emk}_{\am\torusm} \left( \ind^{\am\torusm}_{\am\torusmij}(\varphi'_{\htx,\htz})\right),  \\
 U_d^m & \simeq  \ind_{\am\torusm}^{\emk} \left( \ind_{N_{r}^\delta\torusm}^{\am\torusm} (\triv) \right),
\end{split}
\]
for some $(\htx,\htz) \in \pi^i\com^\times \times \pi^j\com^\times$. Therefore, it is enough to show that $\cW_{ij} \subset \ind_{N_r^\delta\torusm}^{\am\torusm} (\triv)$. Since $\cW_{ij}$ is irreducible the latter is equivalent to
\[
\begin{split}
1&=\dim_\C \Hom_A\left(\ind^{\am\torusm}_{\am\torusmij}(\varphi'_{\htx,\htz}), \ind_{N_{r}^\delta\torusm}^{\am\torusm} (\triv) \right)\\
 &=\dim_\C \left\{f:\am\torusm \to \C \mid f(h_1 gh_2)=\triv(h_1) f(g)\varphi'_{\htxz}(h_2), \text{for all $h_1 \in N_{r}^\delta\torusm $, $g \in \am\torusm$, $h_2 \in \am\torusmij  $}\right\}.
\end{split}
\]
Since $ (\am\torusmij)( {N_{r}^\delta}\torusm)=\am\torusm$, the only candidate for the support of a non-zero intertwining function is the identity element. To see that such non-zero function exists we need to check that $\triv=\varphi'_{\htxz}$ on $ \am\torusmij \cap N_r^\delta\torusm$. Indeed, for a general element in this intersection we have
\[
\varphi_{\htxz}\left(\left[\begin{matrix} x \\ 0 \end{matrix} \right] \rtimes \left[\begin{matrix} \theta_1 & \\  & \theta_2 \end{matrix} \right]\right) = \psi(\htx x)=1,
\]
for case (a) of the proposition, since $x\in \pi^{m-1}\com$ and $\htx\in \pi\com$. Similarly
\[
\varphi_{\htxz}\left(\left[\begin{matrix} 0 \\ z \end{matrix} \right] \rtimes \left[\begin{matrix} \theta_1 & \\  & \theta_2 \end{matrix} \right]\right) = \psi(\htz z)=1,
\]
for case (b) of the proposition, since $z\in \pi^{m-1}\com$ and $\htz\in \pi\com$.
\end{proof}

\begin{proof}[Proof of Theorem~\ref{thm:embedding}] By Lemma \ref{lem:indTA1} and \eqref{eq:def.of.tildeWij} we have
\[
U_c^m = \bigoplus_{0 \le i,j \le m}\widetilde{\cW}_{ij},
\]
and by Proposition~\ref{prop:Udm-subrep-TildeWij}
\[
\sum_{c-\bm \preccurlyeq  d  \prec  c} U_d^m \supset \sum_{(i,j) \ne (0,0)} \widetilde{\cW}_{i,j}.
\]
These together imply that $V_c^m = U_c^m /\sum_{c-\bm \preccurlyeq  d  \prec  c} U_d^m$ embeds in $\widetilde{\cW}_{0,0}$.
\end{proof}


\section{A decomposition}\label{sec:decomposition}

Combining Theorems~\ref{thm:structure.tilde.W00} and~\ref{thm:embedding} we obtain a list of candidate irreducible subrepresentations of $V_c^m$ for $c\in \cone$ and $m \in \mathbb N$ with $\mu(c) \geq m$. In this section we pin down those irreducible representations that indeed occur.
For $m \in \N$ let $\Delta^\delta_m$ be as in Lemma~\ref{lem:stab.of.varphi.in.emk}, namely, isomorphic to $\com^\times$ or $\com$ depending on whether $\delta$ is invertible or not, respectively. To simplify notation we choose such isomorphisms and identify $\dmk$ with either $\com^\times$ or $\com$. Set $\Delta^\delta_0$ to be the trivial group. Let $\iota:\Delta^\delta_m \to \Delta^\delta_{m-1}$ denote reduction modulo $\pi^{m-1}$. The map $\iota$ induces an embedding of characters $\iota^*:\widehat{\Delta}^\delta_{m-1} \hookrightarrow \widehat{\Delta}^\delta_{m}$. The image consists of characters of $\Delta^\delta_m$ which factor through $\Delta^\delta_{m-1}$. Recall that for $\varphi=\varphi_{1,1}$ we have
$\Stab_{\emk}(\varphi)\simeq A \rtimes \Delta^\delta_m$, and that for $\sigma \in \widehat{\Delta}^\delta_{m}$
\[
L_\sigma=\ind_{\am\dmk}^{\emk}(\varphi\sigma).
\]
We also recall that
\[
U_d^m \simeq \ind_{Q_{d-c+m\rho}}^{\emk}(\triv),
\]
and that
\[
Q^\delta_r=
\left[\begin{matrix} \mi^{r_1} \\ \mi^{r_3} \end{matrix}\right] \rtimes \left[\begin{matrix} \com^\times &  \\ \mi^{r_2} & \com^{\times} \end{matrix}\right].
\]

\begin{thm}\label{thm:LsigmaInVcm} For $c \in \cone$, $m \in \mathbb N$ with $\mu(c) \geq m$ and $\sigma \in \widehat{\Delta}^\delta_m$ we have
\[
\mathrm{Hom}_{\emk}(L_\sigma, V_c^m) =\left\{
                                         \begin{array}{ll}
                                           0, & \hbox{if $\sigma \in \iota^*(\widehat{\Delta}_{m-1}^\delta)$;} \\
                                           1, & \hbox{otherwise.}
                                         \end{array}
                                       \right.
\]
\end{thm}

\begin{proof}
We show that $L_\sigma \leqslant \sum_{c-\bm \preccurlyeq d \prec c} U_d^m$ if and only if $\sigma \in \iota^*(\widehat{\Delta}_{m-1}^\delta)$. This is equivalent to the validity of
\begin{equation}\label{eq:condition.IS}\tag*{\qquad \qquad $\mathrm{J}(S)$:}
\begin{split}
\Hom_{\emk}\left(L_\sigma, U_d^m \right)& \ne \left\{0\right\} ~ \textup{for some $d\in S$ if $\sigma \in \iota^*(\widehat{\Delta}_{m-1}^\delta)$} \\
\Hom_{\emk}\left(L_\sigma, U_d^m \right)&= \left\{0\right\} ~ \textup{for all $d\in S$ if $\sigma \not\in \iota^*(\widehat{\Delta}_{m-1}^\delta)$}
\end{split}
\end{equation}
for $S=\{d \mid c-\bm \preccurlyeq d \prec c\}$. Since each of the representations $U_d^m$ ($d \in S$) is contained in one of the three maximal representations indexed by $d \in S'=\{c-\et_1,c-\et_2,c-\et_3\}$, with $\{\et_i\}$ denoting the standard basis, it is enough to prove that $\mathrm{J}(S')$ holds. Using Corollary~\ref{cor:Bd-c} and the definition of $L_\sigma$ we have that for $d \in S'$ and $r=d-(c-\bm)$
\[
\begin{split}
\Hom_{\emk}\left(L_\sigma, U_d^m \right)&=\Hom_{\emk}\left(\ind_{\am\dmk}^{\emk}(\varphi\sigma), \ind_{Q^\delta_{r}}^{\emk} (\triv) \right) \\
 &=\left\{f:\emk \to \C \mid f(h_1 gh_2)=f(g)\varphi\sigma(h_2), ~\forall h_1 \in Q^\delta_{r} , \forall g \in \emk,
 \forall h_2 \in \am\dmk  \right\},
\end{split}
\]
and this space is null if and only if $\varphi\sigma \ne {}^g\triv=\triv$ on $\am\dmk \cap {}^g Q^\delta_{r}$ for all $g \in \emk$. We first observe that for $r=\bm -\et_3=(m,m,m-1)$ this is indeed the case since the subgroup
\[
N^\delta_{\bm-e_3}=\left[ \begin{matrix} 0 \\ \pi^{m-1}\com \end{matrix}\right] \rtimes \left[ \begin{matrix} 1 & 0 \\ 0 & 1 \end{matrix}\right] \subset \am\dmk \cap N^\delta_{\bm-e_3}\torusm
\]
is a characteristic subgroup in $\emk$, and $\varphi$ is nontrivial on it. Since this is true for any $\sigma$ the proof is reduced to the validity of $\mathrm{J}(S'')$ with $S''=\{c-\et_1,c-\et_2\}$.
%
A straightforward computation shows that for the corresponding values of $r$,
and regardless of the value of $\delta$, the elements
\[
g(y)=\left[\begin{matrix} 0 \\0 \end{matrix}\right] \rtimes \left[\begin{matrix} 1 &  \\ y & 1 \end{matrix}\right], y \in \com,
\]
form an exhaustive set of representatives for the double coset space $Q^\delta_{r} \backslash E^\delta / \am \dmk$. Conjugating $Q^\delta_{r}$ with $g(y)$ gives
\begin{equation}
  \begin{array}{ll}
    \left\{\left[\begin{matrix} x \\ \delta x y \end{matrix}\right] \rtimes \left[\begin{matrix} \theta_1 &  \\ (\theta_1-\theta_2)y & \theta_2 \end{matrix}\right] \mid \theta_1,\theta_2 \in \com^\times, x \in \mi^{m-1} \right\}, & \hbox{if $r=\bm-\et_1$;} \\ \\
    \left\{\left[\begin{matrix} 0 \\ 0 \end{matrix}\right] \rtimes \left[\begin{matrix} \theta_1 &  \\ u+(\theta_1-\theta_2)y & \theta_2 \end{matrix}\right] \mid \theta_1,\theta_2 \in \com^\times, u \in \mi^{m-1} \right\} , & \hbox{if $r=\bm-\et_2$.}
  \end{array}
\end{equation}
We now proceed according to the value of $\delta$.
\begin{enumerate}
\item [(a)] ${\delta \in \com^\times}$. \qquad

\noindent The intersection  ${}^{g(y)}Q_{r}^\delta \cap \am\dmk$ is given by
\begin{equation*}
  \begin{array}{ll}
    \left\{\left[\begin{matrix} x \\ \delta x y \end{matrix}\right] \rtimes \left[\begin{matrix} \theta &  \\ (\theta-1)y & 1 \end{matrix}\right] \mid \theta \in 1+\mi^{m-\val(y\delta+1)}, x \in \mi^{m-1} \right\},
& \hbox{if $r=\bm-\et_1$;} \\ \\
    \left\{\left[\begin{matrix} 0 \\ 0 \end{matrix}\right] \rtimes \left[\begin{matrix} \theta &  \\ u+(\theta-1)y & 1 \end{matrix}\right] \mid \theta \in 1+\mi^{m-\val(y\delta+1)-1}, u=(1-\theta)(1+\delta y)\delta^{-1} \right\} , & \hbox{if $r=\bm-\et_2$.}
  \end{array}
\end{equation*}
If $r=\bm-\et_1$, then for a given $\sigma$ the element $g(y)$ supports a nonzero intertwiner if and only if
\[
\varphi\left[\begin{matrix} x \\ \delta xy \end{matrix} \right] \sigma(\theta) = \psi\left(x(1+y\delta)\right)\sigma(\theta)=1, \quad \forall x \in \mi^{m-1}, \forall \theta \in 1+\mi^{m-\val(y\delta+1)},
\]
that is, $\val(1+y\delta)>0$ and $\sigma|_{1+\mi^{m-\val(y\delta+1)}} = \triv$. Therefore, $L_\sigma \leqslant U_{c-\et_1}^m$ if and only if $\sigma|_{1+\mi^{m-1}} = \triv$.

\noindent If $r=\bm-\et_2$, then for a given $\sigma$ the element $g(y)$ supports a nonzero intertwiner if and only if
\[
\varphi\left[\begin{matrix} 0 \\ 0 \end{matrix} \right] \sigma(\theta) = 1, \quad \forall \theta \in 1+\mi^{m-\val(y\delta+1)-1},
\]
that is, $\sigma|_{1+\mi^{m-\val(y\delta+1)-1}} = \triv$. Therefore, $L_\sigma \leqslant U_{c-\et_2}^m$ if and only if $\sigma|_{1+\mi^{m-1}} = \triv$.

\medskip

\item [(b)] ${\delta \in \pi\com}$. \qquad

\noindent The intersection  ${}^{g(y)}Q_{r}^\delta \cap \am\dmk$ is given by
\begin{equation*}
  \begin{array}{ll}
    \left\{\left[\begin{matrix} x \\ 0 \end{matrix}\right] \rtimes \left[\begin{matrix} \theta &  \\ (\theta-1)y & 1 \end{matrix}\right] \mid \theta \in 1+\mi^{m-\val(y\delta+1)}, x \in \mi^{m-1} \right\},
& \hbox{if $r=\bm-\et_1$;} \\ \\
    \left\{\left[\begin{matrix} 0 \\ 0 \end{matrix}\right] \rtimes \left[\begin{matrix} 1 &  \\ u & 1 \end{matrix}\right] \mid  u  \in \mi^{m-1} \right\} , & \hbox{if $r=\bm-\et_2$.}
  \end{array}
\end{equation*}
If $r=\bm-\et_1$, then
\[
\varphi\left[\begin{matrix} x \\ 0 \end{matrix} \right] \ne \triv,
\]
and we conclude that $L_\sigma \nleqslant U_{c-\et_1}^m$ for all $\sigma$'s.

\noindent If $r=\bm-\et_2$, then if $\sigma|_{\mi^{m-1}}=\triv$ in fact all the elements $g(y)$ support a non-zero intertwining function and none of them otherwise. Therefore, $L_\sigma \leqslant U_{c-\et_2}^m$ if and only if $\sigma_{\mi^{m-1}} = \triv$.
\end{enumerate}

The theorem is proved.

\end{proof}

\begin{defn} Let $c \in \cint$ and let $\delta=\pi^{\kinv(c)-\minv(c)}$. For $\sigma \in \widehat{\Delta}^\delta_{\minv(c)} \smallsetminus \iota^*\left(\widehat{\Delta}^\delta_{\minv(c)-1}\right)$ let
\[
\widetilde{L}_{c,\sigma} = \ind_{P_{c-\minv(c)\rho}}^G \left(\mathrm{Ind}_{\emk}^{P_{c-\minv(c)\rho}}\left(L_\sigma\right)\right),
\]
where $\mathrm{Ind}_{\emk}^{P_{c-\minv(c)\rho}}\left(L_\sigma\right)$ stands for the pullback of $L_\sigma$ along the composition of the quotient maps \eqref{eq:bmktoemk} and \eqref{eq:Pcmktobmk}.
\end{defn}

\begin{cor}\label{cor:decomposition.of.Vc} For $c \in \cint$ the decomposition of $V_c$ to irreducible representations is multiplicity free and given by

\[
V_c = \bigoplus_{\sigma \in \widehat{\Delta}^\delta_{\minv(c)} \smallsetminus \iota^*\left(\widehat{\Delta}^\delta_{\minv(c)-1}\right)} \widetilde{L}_{c,\sigma}.
\]

\end{cor}

\begin{proof} Follows from Theorem~\ref{thm:LsigmaInVcm} and statement~\eqref{cn3}.
\end{proof}


\section{Multiplicities and Degrees}
\label{sec:mult.and.deg}

In this section we compute explicitly the dimensions and multiplicities of the irreducible constituents in $V=\ind_B^G(\triv)$.

\subsection{Multiplicities} We first settle a small debt from subsection \ref{subsec:CN}. we shall need the following lemma.

\begin{lem}\label{lem:CN.condition.simplified} Let $x,y,z \in \N$ be such that $0 < x+y -z < \min\{x,y\}$. Then the following are equivalent
\begin{enumerate}
\item $x+y-z \leq \lfloor \min\{x,y \}/2 \rfloor$.
\item $\mu(x,y,z)=\min\{x+y-z,z-x,z-y\}=x+y-z=\kinv(x,y,z)$.
\end{enumerate}
\end{lem}
\begin{proof} Without loss of generality assume that $x \le y$. Then $(1)$ is equivalent to $2(x+y-z) \leq x \leq y$, which in turn is equivalent to $x+y-z \leq z-y \le z-x$, and the latter is equivalent to (2).
\end{proof}

\begin{prop}\label{prop:reformulation} Let $c,d \in \cone$. Then $V_c \simeq V_{d}$ if and only if $\level(c)=\level(d)$, $\kinv(c)=\kinv(d)$, $\minv(c)=\minv(d)$ and $c=d$ if $\kinv(c)>\minv(c)$.
\end{prop}

\begin{proof} We go over the various possibilities for $c$.
\begin{enumerate}
\item [(1)] If $c \in \cone$ satisfies $c_3=c_1+c_2$, then by Theorem 6.1 and Proposition 6.2 in~\cite{MR2504482}, for any $d \in \cone$ we have $V_c \simeq V_d$ if and only if $c_3=d_3$ and $c_1+c_2=d_1+d_2$. These two equations are equivalent to $\level(c)=\level(d)$ and $\kinv(c)=\kinv(d)$, and in such case $0 \le \minv(c)\le \kinv(c)=0$.

\item [(2)] If $c \in \cone$ satisfies $c_3=\max \{c_1,c_2\} \ge 1$, then by~\cite[Theorem 7.1]{MR2504482} and the discussion preceding it, for any $d \in \cone$ we have $V_c \simeq V_d$ if and only if $c=d$. The conditions on $c$ imply that $\kinv(c)>\minv(c)=0$ and the assertion follow.
\item [(3)] It remains to treat $c,d \in \cint$. In particular this means that $\kinv(c),\kinv(d) > 0$. By \cite[Theorem~8.1]{MR2504482}, we have that $V_c \simeq V_d$ if and only if
\begin{enumerate}
\item [(i)] $\level(c)=c_3=d_3=\level(d)$.

\item [(ii)] $\kinv(c)=c_1+c_2-c_3=d_1+d_2-d_3=\kinv(d)$.

\item [(iii)] $c_1+c_2-c_3 \le \lfloor \min\{c_1,c_2,d_1,d_2\}/2 \rfloor$.

\end{enumerate}
In the presence of (i) and (ii), condition (iii) is equivalent to $\minv(c)=\kinv(c)=\kinv(d)=\minv(d)$, by applying Lemma~\ref{lem:CN.condition.simplified} to $c$ and to $d$.

\end{enumerate}

\end{proof}

Define an equivalence relation on $\cone$ by setting $c \sim d$ if and only if $V_c \simeq V_d$. Let $a:\N_0^3 \to \N_0$ be the function defined by
\begin{equation}
a(m,k,\ell)=\left\{
           \begin{array}{ll}
             \ell-3k+1, & \hbox{if $\ell\ge 3k=3m \ge 0$ ;} \\
             1, & \hbox{if $\ell \ge 2m+k>3m \ge 0$,} \\
           \end{array}
         \right.
\end{equation}
and zero otherwise.

\begin{prop}\label{prop:amkl} $|[c]|=a(\minv(c),\kinv(c),\level(c))$ for every $c \in \cone$.
\end{prop}
\begin{proof} If $c \in \cone$ satisfies $\kinv(c)=k=m=\minv(c)$ and $\level(c)=c_3=\ell$ then $k =c_1+c_2-\ell\le \ell-c_1,\ell-c_2$. It follows that $2k \le c_1,c_2 \le \ell-k$, in particular $\ell \ge 3k$, and that the $c$'s satisfying these conditions are precisely of the form
\[
\{(2k+i,\ell-k-i,\ell) \mid i=0,...,\ell-3k\},
\]
and their number is $\ell-3k+1$.

If $c \in \cone$ satisfies $\kinv(c)=k>m=\minv(c)$ and $\level(c)=c_3=\ell$, then Proposition~\ref{prop:reformulation} implies that $V_c \simeq V_d$ if and only if $c=d$ hence $[c]$ is a singleton. We also have that
\[\
\ell=(c_1+c_2-c_3)+(c_3-c_1)+(c_3-c_2) \ge k+2m.
\]
\end{proof}

\begin{cor} A complete decomposition of $V=\ind_B^G(\triv)$ to irreducible representations is given by
\[
V \simeq \left(\bigoplus_{c \in \cbd} V_c\right) \bigoplus \left(    \bigoplus_{[c] \in \cint / \sim} ~ \bigoplus_{\sigma \in  \widehat{\Delta}^\delta_{m} \smallsetminus \iota^*\widehat{\Delta}^\delta_{m-1}}{\widetilde{L}_{c,\sigma}}^{\oplus a(m,k,\ell)}\right).
\]
\end{cor}

\begin{proof} Follows from \eqref{eq:Vc-decomposition}, \eqref{cn2}, Corollary~\ref{cor:decomposition.of.Vc} and Proposition~\ref{prop:amkl}.

\end{proof}

\subsection{Dimensions}
The equidimensionality of the irreducible constituents of $V_c$ implies that the dimension of an irreducible subrepresentation $W$  of $V$ is uniquely determined by the elements $c\in \cone$ with $\mathrm{Hom}_G(W, V_c) \neq (0)$. The dimension of an irreducible representation occurring in $V_c$ ($c \in \cone$) is
\begin{equation}\label{eq:dimensions.of.irreps.in.V}
\left\{
\begin{array}{ll}
1, & \mbox{if $c = (0,0,0)$;} \\
q^2 + q, & \mbox{if $c = (1,0,1), (0,1,1)$;} \\
q^3, &\mbox{if $c = (1,1,1)$;} \\
(1+q^{-1})(1-q^{-3}) q^{2 \level(c)}, & \mbox{if $\kinv(c)=\minv(c), \level(c) \ge 2$;} \\
 (1-q^{-2})(1-q^{-3})q^{2\level(c)+\kinv(c)-\minv(c)}, & \mbox{if $\kinv(c)>\minv(c), \level(c) \ge 2$.}
\end{array}
\right.
\end{equation}
Indeed, the formulae for $\level(c) \le 1$ are well known, see e.g. \cite{MR2504482}. For $c \in \cint$ with $\level(c) \ge 2$ they follows from the construction of $\widetilde{L}_{c,\sigma}$:
\[
\begin{split}
\dim \widetilde{L}_{c,\sigma}&=[\GL_3(\lri):P_{c-\minv(c)\rho}] \cdot \dim L_\sigma \\
&=(1+q^{-1})(1+q^{-1}+q^{-2}) \cdot q^{c_1+c_2+c_3-3\minv(c)}\cdot\left\{
                                                         \begin{array}{ll}
                                                           q^{2\minv(c)}(1-q^{-1}), & \hbox{if $\kinv(c)=\minv(c)$;} \\
                                                           q^{2\minv(c)}(1-q^{-1})^2, & \hbox{if $\kinv(c)>\minv(c)$;}
                                                         \end{array}
                                                       \right.\\
&=\left\{
  \begin{array}{ll}
 (1+q^{-1})(1-q^{-3}) q^{2 \level(c)}, & \mbox{if $\kinv(c)=\minv(c)$;} \\
 (1-q^{-2})(1-q^{-3})q^{2\level(c)+\kinv(c)-\minv(c)}, & \mbox{if $\kinv(c)>\minv(c)$.}
  \end{array}
\right.
\end{split}
\]
For $c \in \cbd$ with $\level(c)\ge 2$ the formulae are given by Theorems 6.1 and 7.1 in \cite{MR2504482}.

\subsection{Uniform representation growth}

Let $\zeta_{G,V}(s)=\sum_{n=1}^{\infty} r_n(G,V) n^{-s}$ with $r_n(G,V)$ the number of irreducible $n$-dimensional subrepresentations of $V$ of dimension $n$. The generating function $\zeta_{G,V}(s)$ enumerates irreducible representations in $V$ according to their dimensions and regardless of their isomorphism type. For every $n \in \N$, let $p_n \in \{0,1,2\}$ be the residue of $n$ modulo $3$ and let $f_n(x),g_n(x) \in \Z[x]$ be the polynomials
\[
\begin{split}
f_n(x) &= \left\{\begin{array}{ll} x^{\lfloor n/6 \rfloor -1 }((p_{\frac{n}{2}}+1) x + (2-p_{\frac{n}{2}})), & \mbox{if $n \in 2 \mathbb N_0 + 4$;} \\ 
 0 , & \mbox{otherwise.}
  \end{array}\right. \\
g_n(x) & =  \frac{x^{\lfloor n/2 \rfloor}  -1 }{x-1} + \frac{x^{\lfloor n/2 \rfloor-1}-1}{x-1} + \min \{ p_n, 1\} x^{\lfloor n/2 \rfloor -1}, \quad (n \ge 5).
\end{split}
\]
Denote $\eta_1(q) = (1+q^{-1})(1-q^{-3})$ and $\eta_2(q) =  (1-q^{-2})(1-q^{-3})$.

\begin{thm}
For every non-archimedean local field $F$ with ring of integers $\lri$ and residue field of cardinality $q$
\[
\zeta_{G,V}(s)=1+2(q+q^2)^{-s}+q^{-3s}+ \sum_{n=4}^\infty f_n(q)\left(\eta_1(q)q^{n}\right)^{-s} + \sum_{n=5}^\infty g_n(q)\left(\eta_2(q)q^{n}\right)^{-s}.
\]

\end{thm}

\begin{proof} The first three terms in $\zeta_{G,V}(s)$ correspond to the subrepresentations of $V_c$ with $\level(c) \in \{0,1\}$. For the remaining, we observe that the dimension of an irreducible representation in $V$ with $\level(c) \ge 2$, as listed in \eqref{eq:dimensions.of.irreps.in.V}, determines whether the representation $V_c$ which contains it has $\kinv(c)=\minv(c)$ or $\kinv(c)>\minv(c)$, hence determines its construction. We formally set $\lri_0^\times=\{1\}$ and $\lri_0=\{0\}$ to be the trivial groups and $\lri_{-1}^\times=\lri_{-1}$ to be the empty set. For the irreducible representations of dimension $\eta_1(q)q^n$ we have
\[
\begin{split}
r_{\eta_1(q) q^{-n}}(G,V)&=\sum_{m~\!\!:~\!\! 0 \le 6m \le 2\ell=n}a(m,m,\ell)|\com^\times \smallsetminus \lri_{m-1}^\times| \\
&=\sum_{ 0 \le m \le  \lfloor \ell/3 \rfloor}(\ell-3m+1)|\com^\times| -\sum_{ -1 \le m \le \lfloor \ell/3 \rfloor -1 }(\ell-3m-4)|\lri_{m}^\times|\\
&=(\ell -3 \lfloor \ell/3\rfloor +1)|\lri_{\lfloor \ell/ 3 \rfloor}^\times| + 3\sum_{ 0 \le m \le  \lfloor \ell/3 \rfloor-1}|\com^\times| \\
&=q^{\lfloor n/6 \rfloor -1 }\left((p_{\frac{n}{2}}+1) q + (2-p_{\frac{n}{2}})\right),
\end{split}
\]
for $n \in 2\N_0+4$ and zero otherwise.

It remains to consider the irreducible representations of dimension $\eta_2(q)q^n$. By \eqref{eq:dimensions.of.irreps.in.V}, an the irreducible subrepresentation of $V$ of dimension $\eta_2(q) q^n$ is contained in some $V_c$ with  $\kinv(c) > \minv(c)$. For $m,n \in \N_0$ let $S(m,n)$ be the (possibly empty) subset of $\cone$
\[
S(m,n)=\left\{ c \in \cone \mid 2 \level(c) + \kinv(c)-\minv(c) = n, \kinv(c) > \minv(c)=m \right\}.
\]
If $S(m,n)$ is non-empty and $c \in S(m,n)$ we have $a(\minv(c),\kinv(c),\level(c)) = 1$, and by Corollary \ref{cor:decomposition.of.Vc} the number of irreducible constituents of $V_c$ is equal to $|\com \smallsetminus \pi \com|$. It follows that the number of irreducible constituents of $V$ of dimension $\eta_2(q) q^n$ is equal to
\begin{eqnarray}
\label{for g}
\sum_{m,k~\!\!:~\!\! n=2\ell+k-m, k > m }a(m,k,\ell)|\com \smallsetminus \pi\com|=\sum_{m \in \N_0} |S(m,n)||\com \smallsetminus\pi \com|.
\end{eqnarray}
Note that by the definition of $S(m,n)$ this is a finite sum for every $n \in \N$. Unraveling definitions and carrying the necessary book keeping yields, for $n \ge 5$,
\[
|S(m,n)|=\left\{
           \begin{array}{ll}
             2\lfloor n/2 \rfloor - 2m, & \hbox{if $0 \le m \le  \lfloor n/2 \rfloor - 1$, $n \not\equiv 0 \mod 3$ ;} \\
             2\lfloor n/2 \rfloor - 2m+1, & \hbox{if $0 \le m \le \lfloor n/2 \rfloor$, $n \equiv 0 \mod 3$,}
           \end{array}
         \right.
\]
and zero otherwise. By substituting these values in (\ref{for g}), we see that for $n \equiv 1, 2 \mod 3$,
\begin{eqnarray}
\sum_{m=0}^{\lfloor n/2 \rfloor - 1} 2\left(\lfloor n/2 \rfloor - m\right)(q-1)q^{m-2} &  =  & 2 \frac{q^{\lfloor n/2 \rfloor} -1 }{q-1},
\end{eqnarray}
and for $ n \equiv 0 \mod 3$,
\begin{eqnarray}
\sum_{m=0}^{\lfloor n/2 \rfloor } \left(2\lfloor n/2 \rfloor - 2m+1\right) (q-1)q^{m-2} &  =  & \frac{q^{\lfloor n/2 \rfloor +1}-1}{q-1} + \frac{q^{\lfloor n/2 \rfloor}-1}{q-1}.
\end{eqnarray}
Combining these expressions, we obtain the result.
\end{proof}

\noindent \textbf{Acknowledgements.} We are most grateful to the referee for valuable comments.


\end{document}